\documentclass[reqno]{siamart171218}
\usepackage{amsmath, amscd, amssymb}
\usepackage{graphicx,epsfig,epstopdf}
\usepackage{enumerate}
\usepackage{url}
\usepackage{algorithmic}

\usepackage{pgfplots}
\usepackage{pgfplotstable}
\pgfplotscreateplotcyclelist{color linestyles*}{%
  {red,solid},%
  {blue,dashed},%
  {black,dotted},%
  {brown,dashdotted},%
  {teal,dashdotdotted}%
}
\pgfplotsset{
  compat = newest,
  legend cell align = left,
  every axis plot/.append style = {
    black
  , mark = none
  , line width = 1pt
  },
  cycle list name = color linestyles*,
  /pgf/number format/1000 sep={},
  ticklabel style={font=\footnotesize},
  label style = {font=\footnotesize},
  every axis legend/.append style = {
    font = \footnotesize
  }
}

\usepackage[vlined,boxruled,linesnumbered,algo2e,algosection]{algorithm2e}
\newcommand{\SetAlgorithmStyle}{
  \setcounter{AlgoLine}{0}
  \SetKwData{Left}{left}\SetKwData{This}{this}\SetKwData{Up}{up}
  \SetKwInOut{Input}{Input}
  \SetKwInOut{Output}{Output}
  \ResetInOut{input}
  \SetKwComment{tcp}{//}{}
  \SetKwFor{For}{for}{}{end}
  \SetArgSty{}
  \DontPrintSemicolon
}

\usepgfplotslibrary{polar}
\usetikzlibrary{arrows}
\tikzset{>=stealth'}

\newcommand*\ind[2]{_{#1}^{(#2)}}
\newcommand{\exclude}[1]{}

\usepackage[textsize=smaller]{todonotes}

\newsiamremark{remark}{Remark}
\newsiamremark{example}{Example}

\newcommand{\ip}[2]{\langle #1, #2\rangle}
\newcommand{\fom}{\mathtt{fom}}
\newcommand{\qfom}{\mathtt{qfom}}
\newcommand{\gmres}{\mathtt{gmres}}
\newcommand{\qgmres}{\mathtt{qgmr}}
\newcommand{\qqgmres}{\mathtt{qqgmr}}
\newcommand{\spec}{\mathrm{spec}}
\newcommand{\conv}{\mathrm{conv}}
\newcommand{\diag}{\mathrm{diag}}
\newcommand{\spann}{\mathrm{span}}
\newcommand{\argmin}{\mathrm{argmin}}
\newcommand{\mathC}{\mathbb{C}}
\newcommand{\mathR}{\mathbb{R}}
\newcommand{\spK}{\mathcal{K}}

\newcommand{\spA}{\mathcal{A}}
\newcommand{\spL}{\mathcal{L}}

\definecolor{SunsetNavy}{HTML}{160A47}
\definecolor{SunsetNavyLL}{HTML}{D0CEDA}

\definecolor{SunsetPurpleD}{HTML}{430254}
\definecolor{SunsetPurple}{HTML}{600479}
\definecolor{SunsetPurpleL}{HTML}{9F68AE}
\definecolor{SunsetPurpleLL}{HTML}{CFB3D6}

\definecolor{SunsetPinkD}{HTML}{6C0A42}
\definecolor{SunsetPink}{HTML}{9B0F5F}
\definecolor{SunsetPinkL}{HTML}{C36F9F}
\definecolor{SunsetPinkLL}{HTML}{E1B7CF}

\definecolor{SunsetRedD}{HTML}{8C121A}
\definecolor{SunsetRed}{HTML}{C81B26}
\definecolor{SunsetRedL}{HTML}{DE767C}
\definecolor{SunsetRedLL}{HTML}{EEBABD}

\definecolor{SunsetOrangeD}{HTML}{A84815}
\definecolor{SunsetOrange}{HTML}{F1671F}
\definecolor{SunsetOrangeL}{HTML}{F59462}
\definecolor{SunsetOrangeLL}{HTML}{F9C2A5}

\newcommand{\af}[1]{\textcolor{SunsetPurple}{#1}}

\newcommand{\cw}[1]{\textcolor{SunsetRedL}{#1}}

\pgfplotstableread{schwingertable_old.dat}{\schwingertableold}
\pgfplotstableread{schwingertable.dat}{\schwingertable}
\pgfplotstableread{schwingertablecond.dat}{\schwingertablecond}
\pgfplotstableread{hltable.dat}{\hltable}
\pgfplotstableread{hltablelarge.dat}{\hltablelarge}
\pgfplotstableread{hltablemgsmall.dat}{\hlmgsmall}
\pgfplotstableread{hltablemglarge.dat}{\hlmglarge}
\pgfplotstableread{hlmgscaling.dat}{\hlmgscaling}

\headers{Krylov type methods exploiting the quadratic numerical range}{A.~Frommer, B.~Jacob, K.~Kahl, C.~Wyss, I.~Zwaan}

\title{Krylov type methods exploiting properties of the quadratic numerical range\thanks{This version dated \today.
\funding{This work was supported in part by Deutsche Forschungsgemeinschaft through the collaborative research centre SFB-TRR55 }}}

\author{Andreas Frommer, Birgit Jacob, Karsten Kahl, Christian Wyss, Ian Zwaan\thanks{Bergische Universit\"at Wuppertal, Wuppertal, Germany \newline
  \mbox{(\{frommer,jacob,kkahl,wyss,zwaan\}@math.uni-wuppertal.de)}.
  }
}

\begin{document}
\maketitle

\begin{abstract} The quadratic numerical range $W^2(A)$ is a subset of the standard numerical range of a linear operator which still contains its spectrum. It arises naturally in operators which 
have a $2 \times 2$ block structure, and it consists of at most two connected components, none of which necessarily convex. The quadratic numerical range can thus reveal spectral gaps, and it can in particular indicate that the spectrum of an operator is bounded away from $0$. 

We exploit this property in the finite-dimensional setting to derive Krylov subspace type methods to solve the system $Ax = b$, in which the iterates arise as solutions of low-dimensional models of the operator whose quadratic numerical ranges is contained in $W^2(A)$. This implies that the iterates are always well-defined and that, as opposed to standard FOM, large variations in the approximation quality of consecutive iterates are avoided, although $0$ lies within the convex hull of the spectrum. We also consider GMRES variants which are obtained in a similar spirit. We derive theoretical results on basic properties of these methods, review methods on how to compute the required bases in a stable manner and present results of several numerical experiments illustrating improvements over standard FOM and GMRES.   
\end{abstract}

\section{Introduction}
It is well known that Krylov subspace methods for a linear system $Ax=b$ with a nonsingular 
matrix $A \in \mathC^{n \times n}$ tend to converge slowly or even diverge or fail in 
situations where $0$ lies in the ``interior'' of the spectrum $\sigma(A)$ of $A$. 
Specifically, if $0$ is contained in the numerical range (or field of values) of $A$, a convex
set which contains $\sigma(A)$, we know that methods based on a Galerkin variational 
characterization like FOM, the full orthogonalization method, can fail due to the 
non-existence of certain iterates which manifests itself numerically by huge variations in 
magnitude and associated stability problems. In methods which are based on residual 
minimization like GMRES, the generalized minimal residual method, stagnation can occur in such 
cases. Related to this, classical convergence theory for Krylov subspace methods, in particular for the non-Hermitian case, typically assumes that $0$ is not contained in the numerical range and then gets quantitative results on convergence speed in which the distance of the numerical range to 0 enters as a parameter, see, e.g., \cite{EisenstatElmanSchultz1983,SaadSchultz1986, Starke97} and the discussion and references in the books \cite{LiesenStrakos2013,Saad2003}.

In this paper we study modifications of the FOM method, and also of GMRES, which converge stably and smoothly when the {\em quadratic numerical range}, a subset of the standard numerical range, splits into two parts which do not contain 0. The quadratic numerical range arises naturally for matrices which have a canonical $2 \times 2$ block structure. Analgously to standard Krylov subspace methods, these modifications are also based on projections. By projecting onto a larger space than the Krylov subspace we manage to preserve the gap in the quadratic numerical range and thus shield the projected matrices away from singularity. At  the same time we do not require more matrix vector multiplications as in standard Krylov subspace methods, i.e.\ one per iteration. 

This paper is organized as follows: Section~\ref{numerical_range:sec} reviews those properties
of the numerical range and the FOM and GMRES method which are important for the sequel. 
Section~\ref{quadratic_fom:sec} first introduces the quadratic numerical range and then 
develops the new modified projection methods termed quadratic FOM and quadratic GMRES. This 
section also contains first elements of an analysis. In Section~\ref{algorithm:sec} we then 
discuss how the new methods can be realized as efficient algorithms before we give some numerical examples in Section ~\ref{numerics:sec}.

\section{Numerical range and FOM} \label{numerical_range:sec}
Regardless of the dimension, $n$, we will always denote by $\langle \cdot,\cdot \rangle$ the standard sesquilinear inner product on $\mathC^n$ and $\| \cdot \|$ the associated norm.
For a linear operator $A \in \mathC^{n \times n}$ the numerical range (or field of values) $W(A)$ is the set of all its Rayleigh quotients 
\[
W(A) = \{ \tfrac{\ip{Ax}{x}}{\ip{x}{x}}: x \in \mathC^n, x \neq 0 \} = \left\{ \ip{Ax}{x}: x \in \mathC^n, \| x \| = 1  \right\} .
\]

$W(A)$ is a compact convex set (see \cite{HornJohnson2013}, e.g.) which contains the spectrum $\spec(A)$. If $A$ is normal, $A^*A = AA^*$, then $W(A)$ is actually the convex hull of $\spec(A)$. For non-normal $A$, the numerical range $W(A)$ can be much larger than the convex hull of the spectrum. If for some $m \leq n$ the matrix $V = [v_1 \mid \cdots \mid v_m] \in \mathC^{n \times m}$ is an orthonormal matrix, i.e. $V^*V = I_m$, the identity on $\mathC^m$, then the numerical range of the ``projected'' matrix $V^*AV \in \mathC^{m \times m}$ is contained in that of $A$, since for all $y \in \mathC^m, y\neq 0$ we have $\ip{y}{y} = \ip{Vy}{Vy}$ and thus
\[
\tfrac{\ip{V^*AVy}{y}}{\ip{y}{y}} = \tfrac{\ip{AVy}{Vy}}{\ip{y}{y}} = \tfrac{\ip{AVy}{Vy}}{\ip{Vy}{Vy}} \in W(A).
\]

For future use we state this observation as a lemma.

\begin{lemma} \label{W_enclosure:lem} Let $A \in \mathC^{n \times n}$ be arbitrary and let $V \in \mathC^{n \times m}$ be orthonormal. Then
\[
W(V^*AV) \subseteq W(A).
\]
\end{lemma}

We continue by summarizing the properties of two Krylov subspace methods, namely FOM \cite{Saad1981} GMRES \cite{SaadSchultz1986}, which are relevant for this work. Proofs and further details can be found in \cite{Saad2003}, e.g. 

A Krylov subspace method for solving the linear system
\[
Ax=b, \enspace A \in \mathC^{n \times n}, b \in \mathC^n,
\]
takes its $k$th iterate from the affine subspace $x^{(0)} + \spK^{(k)}(A,r^{(0)})$, where $r^{(0)} = b-Ax^{(0)}$ and
\[
\spK^{(k)}(A,r^{(0)}) = \spann\{r^{(0)},Ar^{(0)},\ldots,A^{k-1}r^{(0)}\}.
\]
Krylov subspaces are nested and the Arnoldi process (see \cite{Saad2003}, e.g.), iteratively computes an orthonormal basis $v^{(1)},v^{(2)},\ldots $ for 
these subspaces. Collecting the vectors into an orthonormal matrix $V^{(k)} = [v^{(1)} \mid \cdots \mid v^{(k)}]$, the Arnoldi process can be
summarized by the Arnoldi relation
\begin{equation} \label{Arnoldi_relation:eq}
AV^{(k)} = V^{(k+1)} \underline{H}^{(k)}, k=1,2,\ldots.
\end{equation}
where $\underline{H}^{(k)}  \in \mathC^{(k+1) \times k}$ collects the coefficients resulting from the orthonormalization process. It has upper Hessenberg structure. Denoting by $H^{(k)}$ the $k \times k$ matrix obtained from $\underline{H}^{(k)}$ by removing the last row, we see that
\[
H^{(k)} = (V^{(k)})^*AV^{(k)}.
\]
The full orthogonalization method (FOM) is the Krylov subspace method with iterate $x^{(k)}_{\fom}$ characterized variationally via
\[
x^{(k)}_{\fom} \in x^{(0)} + \spK^{(k)}(A,r^{(0)}), \enspace r^{(k)}_\fom = b-Ax^{(k)}_{\fom} \perp \spK^{(k)}(A,r^{(0)}),
\]
which gives
\[
x^{(k)}_{\fom} = x^{(0)} + V^{(k)} (H^{(k)})^{-1} (V^{(k)})^*  r^{(0)},
\]
provided $H^{(k)}$ is nonsingular. Note that since $v_1$ is a multiple of $r^{(0)}$ we have
\begin{equation} \label{e1:eq}
(V^{(k)})^*r^{(0)} = \|r^{(0)}\|e_1^k, 
\end{equation}
where $e_1^k$ denotes the first canonical unit vector in $\mathbb C^k$. 

For an arbitrary (nonsingular) matrix $A$, the matrix $H^{(k)}$ can become singular in which case the $k$-th FOM iterate does not exist.  
An important consequence of Lemma~\ref{W_enclosure:lem} is therefore that such a breakdown of FOM cannot occur if $0 \not \in W(A)$, and, moreover, that $H^{(k)}$ will have no eigenvalues with modulus smaller than the distance of $W(A)$ to $0$. On the other hand, if $0 \in W(A)$, even when $H^{(k)}$ is nonsingular, it can become arbitrarily ill-conditioned, which then typically yields large residuals for the corresponding iterates and which is observed in practice as irregular convergence behavior. 

We can interprete FOM as the method which for each $k$ builds a reduced model $H^{(k)}$ of dimension $k$ of the original matrix and 
then obtains its iterate $x^{(k)}_{\fom}$ by lifting the solution of the corresponding reduced system $H^{(k)} \xi_k = (V^{(k)})^*r^{(0)}$ back to the full space as a correction to the initial guess $x^{(0)}$, $x^{(k)}_{\fom} = x^{(0)} + V^{(k)}\xi_k$. This interpretation will serve as a guideline for our development of the ``quadratic'' FOM method in section~\ref{quadratic_fom:sec}. 

The generalized minimal residual method (GMRES) is the Krylov subspace meth\-od with iterate $x^{(k)}_{\gmres}$ characterized variationally via
\[
x^{(k)}_{\gmres} \in x^{(0)} + \spK^{(k)}(A,r^{(0)}), \enspace  r^{(k)}_\gmres = b-Ax^{(k)}_{\gmres} \perp A \cdot \spK^{(k)}(A,r^{(0)}),
\]
This implies that the residual $b-Ax^{(k)}_{\gmres}$ is smallest in norm among all possible residuals $b-Ax$ with $x  \in x^{(0)} + \spK^{(k)}(A,r^{(0)})$, i.e.\ $x^{(k)}_\gmres$ solves the least squares problem
\[
x^{(k)}_\gmres = \argmin_{x \in x^{(0)}+\mathcal{K}^{(k)}(A,r^{(0)})} \| b - Ax\| = x^{(0)} + \argmin_{y \in \mathcal{K}^{(k)}(A,r^{(0)})} \| r^{(0)}-Ay \|.
\]
To obtain an efficient algorithm it is important to see that this $n \times k$ least squares problem can be reduced 
to a $(k+1)\times k$ system due to the Arnoldi relation \eqref{Arnoldi_relation:eq}: We have that $x^{(k)}_{\gmres} = x\ind{}{0}+V^{(k)} \xi^{(k)}$ where $\xi^{(k)}$ solves 
\begin{equation} \label{reduced_ls:eq}
\xi^{(k)}= \argmin_{\xi \in \mathC^k} \| (V^{(k+1)})^* r^{(0)} - \underline{H}^{(k)} \xi \|,
\end{equation}
where $(V^{(k+1))^*}r^{(0)} = \|r^{(0)}\|e_1^{k+1}$. 

In case that $H^{(k)}$ is nonsingular, one can use the normal equation for \eqref{reduced_ls:eq} to characterize
$\xi_k = (\hat{H}^{(k)})^{-1} e_1^k$, where
\begin{equation} \label{reduced_model_gmres:eq}
\hat{H}^{(k)} = H^{(k)} + |h_{k+1,k}|^2((H^{(k)})^{-*}e_k)e_k^*, \mbox { where } h_{k+1,k} \mbox{ is the } (k+1,k) \mbox{ entry  of } \underline{H}^{(k)}.
\end{equation}
This means that the GMRES approach constructs a reduced model $\hat{H}^{(k)}$ which differs by the FOM model by a matrix of rank 1.
The eigenvalues of $\hat{H}^{(k)}$ are called the  
 {\em harmonic Ritz values} of $A$ w.r.t.\ $\spK^{(k)}(A,r^{(0)})$, i.e.\ the values $\mu$ for which 
\[
A^{-1}x - \tfrac{1}{\mu}x  \perp A\spK^{(k)}(A,r^{(0)}) \enspace \mbox{ for some } x \in A\spK^{(k)}(A,r^{(0)}), x \neq 0.
\]
They are the inverses of the Ritz values of $A^{-1}$ w.r.t\ the subspace $A\spK(A,r^{(0)})$ which implies
\[
\mu^{-1} \in W(A^{-1}).
\]
With $\rho$ denoting the numerical radius of $A^{-1}$, i.e.\ $\rho = \max \{ | \omega |: \omega \in W(A^{-1}) \}$ we see that 
$|\mu | \geq \rho^{-1}$. In this sense, as opposed to FOM, the GMRES approach shields the eigenvalues of the reduced model $\hat{H}^{(k)}$ away from $0$. Note that if $H^{(k)}$ is singular, GMRES stagnates, i.e.\ $x^{(k)}_\gmres = x^{(k-1)}_\gmres$.

\section{Quadratic numerical range, QFOM and QGMRES} \label{quadratic_fom:sec}
We now assume that $A  \in \mathC^{n \times n}$ has a ``natural'' block decomposition of the form 
\begin{equation} \label{A_block_structure:eq}
A  = \begin{bmatrix} A_{11} & A_{12} \\ A_{21} & A_{22} \end{bmatrix} \enspace \mbox{ with } A_{ij} \in \mathC^{n_i \times n_j}, i,j=1,2, \, n_1+n_2 = n, \, n_1,n_2\ge 1.
\end{equation}
All vectors $x$ from $\mathC^{n}$ are endowed with the same block structure
\[
  x = \begin{bmatrix} x_1 \\ x_2 \end{bmatrix}, \enspace x_i \in \mathC^{n_i}, i=1,2.
\]
The definition of the quadratic numerical range goes back to \cite{LangerTretter1998}, where it was introduced as a tool to
localize spectra of block operators in Hilbert space. 

\begin{definition} \label{ranges:def}
The quadratic numerical range $W^2$ of $A$ is given as 
\[
W^2(A) = \bigcup_{\|x_1\| = \|x_2\| = 1} \spec\left( \begin{bmatrix} x_1^*A_{11}x_1 & x_1^*A_{12}x_2 \\ x_2^*A_{21}x_1 & x_2^*A_{22}x_2   \end{bmatrix} \right).
\]
\end {definition}

The following basic properties are, e.g., proved in \cite{TretterBook2008}

\begin{lemma} \label{W2_properties:lem} We have
\begin{itemize}
\item[(i)] $W^2(A)$ is compact,
\item[(ii)] $W^2(A)$ has at most two connected components, 
\item[(iii)] $\spec(A) \subseteq W^2(A) \subseteq W (A),$
\item[(iv)] If $n_1, n_2\ge 2$, then $W(A_{11}), W(A_{22}) \subseteq W^2(A)$.
\end{itemize}
\end{lemma}

The following counterpart of Lemma~\ref{W_enclosure:lem} holds.

\begin{lemma} \label{W2_enclosure:lem} Let $A \in \mathC^{n \times n}$ have block structure \eqref{A_block_structure:eq} and assume that $V_1 \in \mathC^{n_1 \times m_1}, V_2 \in \mathC^{n_2 \times m_2}$ with $m_i \leq n_i, i=1,2$ have orthonormal columns. Put $V = [\begin{smallmatrix} V_1 & 0 \\ 0 & V_2 \end {smallmatrix} ] \in \mathC^{n \times m}$ with $m=m_1+m_2$. Then
\[
W^2(V^*AV) \subseteq W^2(A), \mbox{ where } V^*AV = \begin{bmatrix} V_1^*A_{11}V_1 & V_1^*A_{12}V_2 \\ V_2^*A_{21}V_1 & V_2^*A_{22}V_2 \end{bmatrix} \in \mathC^{m \times m}.
\]
\end{lemma}
\begin{proof}
Let $y_i \in \mathC^{m_i}$ for $i=1,2$ with $\|y_i\| = 1$. Then $ x_i := V_{i}y_i$ 
satisfies $\|x_i\| = 1, i=1,2$, and since
\[
\begin{bmatrix} (y_1)^*V_1^*A_{11}V_1 y_1 & (y_1)^*V_1^*A_{12}V_2 y_2 \\ (y_2)^*V_2^*A_{21}V_1 y_1 & (y_2)^*V_2^*A_{22}V_2 y_2\end{bmatrix}
\ = \ 
\begin{bmatrix} x_1^*A_{11}x_1 & x_1^*A_{12}x_2 \\ x_2^*A_{21}x_1 & x_2^*A_{22}x_2   \end{bmatrix}
\]
we obtain $W^2(V^*AV) \subseteq W^2(A)$.
\end{proof}

Our approach is now to build a Krylov subspace type method where, as opposed to FOM, the iterates are obtained by inverting a reduced model of $A$ whose quadratic numerical range is contained in that of $A$. In this manner, if $0 \not \in W^2(A)$ with $\delta =
\min \{|\mu|: \mu \in W^2(A) \}$ denoting the distance of $0$ to $W^2(A)$, no eigenvalue of the reduced model will have modulus smaller than $\delta$. In cases where $0 \in W (A)$ and $0 \not \in W^2(A)$ this bears the potential of obtaining smoother and faster convergence than with FOM and, as it will turn out experimentally, also faster than with GMRES.

We project the Krylov subspace $\spK^{(k)}(A,r^{(0)})$ onto its first $n_1$ and last $n_2$ components, 
respectivley, 
denoted $\spK^{(k)}_1(A,r^{(0)}) \subseteq \mathC^{n_1}$ and $\spK^{(k)}_2(A,r^{(0)}) \subseteq \mathC^{n_2}$. Clearly,
\[
\spK^{(k)}(A,r^{(0)}) \subseteq \spK^{(k)}_1(A,r^{(0)}) \times \spK^{(k)}_2(A,r^{(0)}) =: \spK^{(k)}_\times(A,r^{(0)}),
\]
and $\dim \spK^{(k)}(A,r\ind{}{0}) \leq \dim \spK_{\times}^{(k)}(A,r^{(0)}) =: d_\times^{(k)} \leq 2k$. Note that the dimension $d_i^{(k)}$ of either $\spK^{(k)}_i(A,r^{(0)})$ may be less than $k$ and that $d\ind{\times}{k} = d\ind{1}{k} + d\ind{2}{k}$.

We can obtain an orthonormal basis for each of the $\spK\ind{i}{k}(A,r\ind{}{0})$ as the columns of the matrix $V\ind{i}{k}$ which arises from the QR-decomposition of the respective block of the matrix $V\ind{}{k}$ from the Arnoldi process, i.e.\ 
\begin{equation} \label{V_relations:eq}
V\ind{}{k}\!=\! \begin{bmatrix} V\ind{1}{k}R\ind{1}{k} \\ V\ind{2}{k}R\ind{2}{k} \end{bmatrix}, 
 V\ind{i}{k} \in \mathC^{n_i \times d\ind{i}{k}} \mbox{ orthonorm.}, \, R\ind{i}{k} \in \mathC^{d\ind{i}{k} \times k} \mbox{ upper triang.} 
\end{equation}

Note that with this definition of $V\ind{i}{k}$ we have the useful property that $V\ind{i}{k+1}$ arises from 
$V\ind{i}{k}$ by the addition of a new last column, just in the way $V\ind{}{k+1}$ arises from $V\ind{}{k}$, with the exception that the new last column could be empty, i.e.\ there is no new last column, when the last column of the $i$th block in $V^{(k)}$ is linearly dependent of the other columns. Similarly $R\ind{i}{k+1}$ arises from $R\ind{i}{k}$ by adding a new last column and a new last row (if it is not empty). 

We now introduce variational characterizations based on the space $\spK_{\times}^{(k)}(A,r^{(0)}).$

\subsection{QFOM}
{\em Quadratic FOM} imposes a Galerkin condition using  $\spK_{\times}^{(k)}(A,r^{(0)})$.

\begin{definition} \label{qfom:def} The $k$-th {\em quadratic FOM (``QFOM'') iterate} $x_{\qfom}^{(k)}$ is defined variationally through
\begin{equation} \label{qfom_variational:eq} 
x_{\qfom}^{(k)} \in x^{(0)} + \spK^{(k)}_\times(A,r^{(0)}),\enspace b-A x_{\qfom}^{(k)} \perp \spK_{(\times)}^{(k)}(A,r^{(0)}).
\end{equation}
\end{definition}

The columns of the matrix
\[
V_\times^{(k)} = \begin{bmatrix} V^{(k)}_1 & 0 \\ 0 &  V^{(k)}_2 \end{bmatrix}
\]
form an orthonormal basis of $\spK^{(k)}_\times(A,r^{(0)})$. Defining the reduced model $H_\times^{(k)}$ of $A$ as
\begin{equation} \label {W2_model:eq}
H_\times^{(k)} =  (V\ind{\times}{k})^* A V\ind{\times}{k}  = \begin{bmatrix} (V^{(k)}_1)^*A_{11}V^{(k)}_1 & (V^{(k)}_1)^*A_{12}V^{(k)}_2 \\ (V^{(k)}_2)^*A_{21}V^{(k)}_1 & (V^{(k)}_2)^*A_{22}V^{(k)}_2 \end{bmatrix}
\end{equation}
we see that if $H_\times^{(k)}$ is nonsingular, the QFOM iterate $x^{(k)}_{\qfom}$ according to Definition~\ref{qfom:def} exists and can be represented as 
\begin{equation} \label{qfom_representation:eq}
x_{\qfom}^{(k)} = x^{(0)} + V_\times^{(k)} (H_\times^{(k)})^{-1}  (V_\times^{(k)})^*   r^{(0)}. 
\end{equation}

Instead of \eqref{e1:eq} we now have
\begin{equation} \label{e2:eq}
(V_\times^{(k)})^*   r^{(0)} = \begin{bmatrix} \|r^{(0)}_1\| e^{d^{(k)}_1}_1 \\ \|r^{(0)}_2\|e^{d^{(k)}_2}_1 \end{bmatrix}, \mbox{ where } r^{(0)} = \begin{bmatrix} r^{(0)}_1 \\ r^{(0)}_2 \end{bmatrix}.
\end{equation}
If $H_\times^{(k)}$ is singular, the $k$-th QFOM iterate does not exist.
We will show in section~\ref{algorithm:sec} that computing $x^{(k)}_{\qfom}$ costs $k$ matrix-vector multiplications with $A$ plus additional arithmetic operations of order $\mathcal{O}(k^3)$. The cost is therefore the same as for standard FOM in terms of matrix-vector multiplications, and the additional cost is also of the same order (though with a larger constant).   

\subsection{Analysis of QFOM}

The following theorem summarizes some basic properties of QFOM. 

Recall that the {\em grade} of a vector $v$  with respect to a square matrix $A$  is the first index $g(v)$ for which $\spK^{(g(v))}(A,v) = \spK^{(g(v)+1)}(A,v)$. We know (see \cite{Saad2003}, e.g.) that 
then $\spK^{(g(v))}(A,v) = \spK^{(g(v)+i)}(A,v)$ for all $i \geq 0$ and that $A^{-1}v \in \spK^{(g(v))}(A,v)$, provided $A$ is nonsingular.

\begin{theorem} \label{QFOM_properties:thm} Let $A$ be nonsingular. 
Then
\begin{itemize}
  \item[(i)] [Finite termination]  There exists an index $k_{\max} \leq g(r\ind{}{0})$ such that $A^{-1}r\ind{}{0} \in \spK\ind{\times}{k_{\max}}(A,r\ind{}{0})$, and if $H^{(k_{\max})}_\times$ is nonsingular, $x_{\qfom}^{(k_{\max})}$ exists and $x_\qfom^{(k_{\max})} = A^{-1}b$.
  \item[(ii)] [Quadratic numerical range property] The inclusion $W^2(H_\times^{(k)} ) \subseteq W^2(A)$ holds for $k=1,\ldots,k_{\max}$, where the $2 \times 2$ block structure of $H_\times^{(k)}$ is given in \eqref{W2_model:eq}.
  \item[(iii)] [Existence] If $0 \not \in W^2(A)$, then $x_\qfom^{(k)}$ exists for $k=1,\ldots,k_{\max}$, i.e.\ $H_{k}^\times$ is nonsingular for 
  all $k=1,\ldots,k_{\max}$.
  \end{itemize}
\end{theorem}
\begin{proof}
To show (i), let $g$ be the grade of $r\ind{}{0}$ w.r.t.\ $A$ and let $k_{\max}\leq g$ be the smallest index $k$ for which $\spK\ind{}{g}(A,r\ind{}{0}) \subseteq \spK\ind{\times}{k}(A,r\ind{}{0})$. Since $A$ is nonsingular, there exists $y^* \in \spK\ind{\times}{k_{\max}}(A,r^{(0)})$ with $Ay^* = r\ind{}{0}$, i.e.\ $y^* = A^{-1}r\ind{}{0}$. As a consequence, $x^* = A^{-1}b = x\ind{}{0}+ y^* \in x\ind{}{0} + \spK\ind{\times}{k_{\max}}(A,r\ind{}0{})$  
%
satisfies the variational characterization~\eqref{qfom_variational:eq} from Definition~\ref{qfom:def} just as $x\ind{\qfom}{k_{\max}}$ does. If  $H^{(k_{\max})}_\times$ is
nonsingular there is exactly one vector from $x^{(0)}+\spK^{(k_{\max})}_{\times}(A,r^{(0)})$ which satisfies \eqref{qfom_variational:eq} which gives $x_{\qfom}^{(k_{\max})} = x^*$.

Part (ii) follows directly from Lemma~\ref{W2_enclosure:lem}. Finally, part (iii) is an immediate consequence of part (ii) and the spectral enclosure property stated as Lemma~\ref{W2_properties:lem}(iii).
\end{proof}

More far-reaching results seem to be difficult to obtain. In particular, the absence of a polynomial interpolation property---which we discuss in the sequel---makes it impossible to follow established 
concepts from standard Krylov subspace theory.  

The FOM iterates satisfy a polynomial interpolation property: We know that $(H^{(k)})^{-1} = q(H^{(k)})$ where $q$ is the polynomial of degree at most $k-1$ which interpolates the function $z \to z^{-1}$ on the eigenvalues in the Hermite sense, i.e.\ up to the $j-1$st deriviative if the multiplicity of the eigenvalue in the minimal polynomial is $j$; see \cite{Higham2008}. We have that
\[
V^{(k)}(H^{(k)})^{-1}(V^{(k)})^*r^{(0)} = V^{(k)}q(H^{(k)})(V^{(k)})^*r^{(0)} = q(A) r^{(0)},
\]
where the last, important equality holds because $V^{(k)}(V^{(k)})^*$ represents the orthogonal projector on $\spK_m(A,r^{(0)})$, thus implying that for all powers $j=0,\ldots,k-1$ we have $V^{(k)} (H^{(k)})^j (V^{(k)})^* r^{(0)} = V^{(k)}((V^{(k)})^*AV^{(k)})^j(V^{(k)})^*r^{(0)} = A^j r^{(0)}$. As a consequence
\begin{equation} \label{interpolation_property:eq}
x^{(k)}_\fom = x^{(0)} + q(A)r^{(0)}.
\end{equation}
Since $\hat{H}^{(k)}$ differs from $H^{(k)}$ only in its last column, the same argument as above shows that 
an analogue of \eqref{interpolation_property:eq} holds for the GMRES iterates, where now $q$ interpolates on the spectrum of $\hat{H}^{(k)}$. This interpolation property is very helpful in the analysis of the FOM and GMRES method, but there is no analog for QFOM. Indeed, while we can express
$(H_\times^{(k)})^{-1}$ as a polynomial $q$ of degree at most $d_1^{(k)} + d_2^{(k)} -1 \leq 2k-1$ in $H_\times^{(k)}$, the matrix $V_\times^{(k)} (V_\times^{(k)})^*$ is an orthogonal projector on $K_\times^{(k)}(A,r^{(0)})$ which contains $K^{(k)}(A,r^{(0)})$ but not necessarily the higher powers $A^i r^{(0)}$ for $i \geq k$. Therefore,
we cannot conclude that $(V_\times^{(k)}) (H_\times^{(k)})^i (V_\times^{(k)})^*r^{(0)} = (V_\times^{(k)}) ((V_\times^{(k)})^*AV_\times^{(k)})^i (V_\times^{(k)})^*r^{(0)}$
would be equal to $A^ir^{(0)}$ for $i \geq k$, and therefore, since the degree of the polyonomial $q$ is likely 
to be larger than $k-1$ don't get $V_\times^{(k)} q(H_\times^{(k)})(V_\times^{(k)})^*r^{(0)} = 
q(A)r^{(0)}$.

To finish this section, we look at the very extreme case in which $W^2(A)$ consists of just one or two points, and we show that in this case QFOM obtains the solution after just one iteration in a larger number of cases than standard FOM or GMRES does. 
So assume $W^2(A) = \{ \lambda_1,\lambda_2\}$, where $\lambda_1 = \lambda_2$ is allowed. 

\begin{lemma} \label{W2_extreme:lem} 
Let $n_1, n_2\ge 2$.
$W^2(A) = \{\lambda_1, \lambda_2\}$ iff 
\begin{equation} \label{form_A:eq}
A = \begin{bmatrix} \lambda_1 I & A_{12} \\A_{21} & \lambda_2 I \end{bmatrix}, \mbox{ where } A_{12} = 0 \mbox{ or } A_{21} = 0,
\end{equation}
(up to a permutation of $\lambda_1$, $\lambda_2$ on the diagonal).
\end{lemma}
\begin{proof} For $x_i \in \mathC^{n_i}, \|x_i\| = 1, i=1,2$ denote 
\[
\alpha = x_1^*A_{11}x_1, \beta = x_1^*A_{12}x_2, \gamma = x_2^*A_{21}x_1, \delta = x_2^*A_{22}x_2.
\]
Then $\lambda \in W^2(A)$ iff 
\begin{equation} \label{quadratic:eq}
(\lambda-\alpha)(\lambda-\delta)-\beta\gamma = 0
\end{equation}
for $\alpha,\beta,\gamma,\delta$ associated with such $x_1,x_2$.
Now, if $A$ is of the form~\eqref{form_A:eq}, then $\beta\gamma = 0$, $\alpha=\lambda_1$ and $\delta=\lambda_2$, which immediately gives that  \eqref{form_A:eq} is sufficient to get $W^2(A) = \{ \lambda_1,\lambda_2\}$.

To prove necessity, assume $W^2(A)= \{\lambda_1, \lambda_2\}$. Since $W(A_{ii}) \subseteq W^2(A)$ for $i=1,2$ by Lemma~\ref{W2_properties:lem}(iv) and since the numerical range is convex, this implies $W(A_{11})=\{\mu_1\}$, $W(A_{22})=\{\mu_2\}$ with $\mu_1,\mu_2\in \{\lambda_1,\lambda_2\}$. Consequently $A_{11} = \mu_1 I, A_{22} = \mu_2 I$.
For a proof by contradiction assume now that both $A_{12}$ and $A_{21}$ are nonzero. Then there exist normalized vectors $x_1,x_2, y_1,y_2$ such that $x_1^*A_{12}x_2 \neq 0$ and $y_2^*A_{21} y_1 \neq 0$. For $\epsilon \in \mathR$, consider $z_1 = x_1 + \epsilon y_1, z_2 = x_2+ \epsilon y_2$. Then $z_1^*A_{12}z_2 \neq 0$ for $\epsilon \neq 0$ small enough and 
\[
z_2^*A_{21}z_1 = x_2^*A_{21}x_1 + \epsilon (x_2^*A_{21}y_1 + y_2^*A_{21}x_1) + \epsilon^2 y_2^*A_{21}y_1.  
\]
This quadratic function in $\epsilon$ is nonzero for sufficiently small $\epsilon \neq 0$. 
Thus, for $\epsilon \neq 0$ sufficiently small, taking the normalized versions of $z_1,z_2$ we get that the corresponding $\beta$ and $\gamma$ are both nonzero. 
Consequently the expression
\[ (\lambda-\mu_1)(\lambda-\mu_2)-\beta\gamma\]
is nonzero for $\lambda=\mu_1\in W^2(A)$, but zero at the same time by \eqref{quadratic:eq}.
Thus at least one of the matrices $A_{12}, A_{21}$ is zero. 
It follows that $W^2(A)=\{\mu_1,\mu_2\}$ and consequently
$\mu_1=\lambda_1$ and $\mu_2=\lambda_2$ up to a permutation of $\lambda_1$, $\lambda_2$.
\end{proof}

With these preparations we obtain the following result. 

\begin{theorem} Assume that $n_1, n_2\ge 2$ and $0\notin W^2(A) = \{\lambda_1,\lambda_2\}$ and consider the linear system
\[
A x = b.
\]
Without loss of generality we assume that iterations start with the initial guess $x^{(0)} = 0$.
We also denote by $x^{*} = A^{-1}b$ the solution of the system. Then
\begin{itemize}
\item[(i)] $x_\fom^{(1)} = x^{*}$ if $b$ is an eigenvector of $A$. In all other cases, $x_\fom^{(2)} = x^{*}$.
\item[(ii)] $x_\qfom^{(1)} = x^{*}$ if $A_ {12}b_2$ is collinear to $b_1$ (or $0$). In all other cases, $x_\qfom^{(2)} = x^{*}$. 
\end{itemize}
\end {theorem}
\begin{proof} 
By Lemma~\ref{W2_extreme:lem} we know that $A$ has the form 
\[
A = \begin{bmatrix} \lambda_1 I & A_{12} \\ 0 & \lambda_2 I \end{bmatrix} \mbox{ or }
A = \begin{bmatrix} \lambda_1 I & 0 \\ A_{21} & \lambda_2 I \end{bmatrix},
\]
and we focus on the first case. The second case can be treated in a completely analogous manner.
We first note that if $\lambda_1 \neq \lambda_2$, the eigenvectors to the eigenvalue $\lambda_1$ are of the form $\left[ \begin{smallmatrix} x_1 \\ 0  \end{smallmatrix} \right]$ and the eigenvectors to the eigenvalue $\lambda_2$ are given by $\left[ \begin{smallmatrix}  (\lambda_2-\lambda_1)^{-1}A_{12}x_2 \\  x_2  \end{smallmatrix} \right]$. If $\lambda_1 = \lambda_2$, all vectors of the form 
$\left[ \begin{smallmatrix}  x_1 \\ x_2 \end{smallmatrix} \right]   $ with $A_{12}x_2 = 0$ are eigenvectors.  The theorem thus asserts that the situations where FOM gets the solution in the first iteration is a true subset of the situations in which QFOM obtains the solution in its first iteration. 

To proceed, we observe that the minimal polynomial of $A$ is $p(z) = (z-\lambda_1)(z-\lambda_2)$ in all cases except for the case where $\lambda_1 = \lambda_2$ and $A_{12} = 0$, i.e.\ when $A = \lambda_1 I$ with minimal polynomial $p(z) = (z-\lambda_1)$. Since $x_\fom^{(1)} \in \spK^{(1)}(A,b)$, which is spanned by $b$, FOM obtains the solution $x^{*}$ in the first iteration exactly in the case where $b$ is an eigenvector of $A$. If $b$ is not an eigenvector of $A$, then the minimal polynomial is $p(z) = (z-\lambda_1)(z-\lambda_2)$ so that the grade of $b$ is 2, and FOM obtains the solution §$x^{*}$ in its second iteration. 


If $b_1 \neq 0$ and $b_2 \neq 0$, the first iteration of QFOM obtains $x^{(1)}_\qfom$ as 
\begin{eqnarray*}
x^{(1)}_\qfom &=& \begin{bmatrix} \tfrac{1}{\|b_1\|}b_1 & 0 \\ 0 & \tfrac{1}{\|b_2\|}b_2 \end{bmatrix}
\begin{bmatrix} \lambda_1 & \tfrac{1}{\|b_1\|\, \|b_2\|}b_1^*A_{12}b_2 \\ 0 & \lambda_2  \end{bmatrix}^{-1} \begin{bmatrix} {\|b_1\|} \\ {\|b_2\|} \end{bmatrix} \\
&=& 
\begin{bmatrix} \tfrac{1}{\|b_1\|}b_1 & 0 \\ 0 & \tfrac{1}{\|b_2\|}b_2 \end{bmatrix}
\begin{bmatrix} \tfrac{1}{\lambda_1} & -\tfrac{1}{\lambda_1\lambda_2}\tfrac{1}{\|b_1\|\, \|b_2\|}b_1^*A_{12}b_2 \\ 0 & \tfrac{1}{\lambda_2}  \end{bmatrix} \begin{bmatrix} {\|b_1\|} \\ {\|b_2\|} \end{bmatrix} \\
&=& 
\begin{bmatrix} \tfrac{1}{\lambda_1}(b_1-\tfrac{1}{\lambda_2}\tfrac{1}{\|b_1\|^2}b_1b_1^*A_{12}b_2) \\ \tfrac{1}{\lambda_2} b_2 \end{bmatrix},
\end{eqnarray*}
which is equal to the solution 
\[
x^{*} = \left[ \begin{matrix} \tfrac{1}{\lambda_1}(b_1-\tfrac{1}{\lambda_2}A_{12}b_2) \\ \tfrac{1}{\lambda_2} b_2 \end{matrix} \right]
\]
exactly when the projector $\tfrac{1}{\|b_1\|^2}b_1b_1^*$ acts as the identity on $A_{12}b_2$, i.e.\ when $A_{12}b_2$ is zero or collinear to $b_1$. A similar observation holds if $b_1 = 0$ or $b_2 = 0$. In all other cases, by Theorem~\ref{QFOM_properties:thm} we have $x^{(2)}_\qfom = x^{*}$ since the grade of $b$ then equals 2. 
\end{proof}

\subsection{QGMRES and QQGMRES}
In principle, we can proceed in a manner similar to QFOM to derive a ``quadratic'' GMRES method. Variationally, its iterates $x\ind{\qgmres}{k}$ would be characterized by
\begin{equation} \label{full_qgmres:eq}
x\ind{\qgmres}{k} \in x^{(0)} + \spK^{(k)}_\times(A,r^{(0)}), \enspace b-Ax\ind{\qgmres}{k} \perp A \spK^{(k)}_\times(A, r^{(0)}),
\end{equation}
which is equivalent to minimizing the norm of the residual $\|b-Ax\|$ for $x \in x^{(0)} + K_\times^{(k)}(A,r^{(0)})$. Thus, as for standard GMRES, we can get $x^{(k)}_\qgmres$ as  
$x^{(0)} + V_\times^{(k)} \eta_k$ where $\eta_k$ solves the least squares problem
\begin{equation} \label{ls_qgmres:eq}
\eta_k = \argmin_{\eta \in \mathC^{d\ind{\times}{k}}} \| r^{(0)}- AV^{(k)}_\times\eta \|.
\end{equation}
However, as opposed to standard GMRES, it is not possible to recast this $n \times d\ind{\times}{k}$ least squares problem into one with a reduced 
first dimension, since an analogon to the Arnoldi relation \eqref{Arnoldi_relation:eq} does not hold for the product spaces $\mathcal{K}_\times^{(k)}(A,r^{(0)})$. In particular, for $x^{(k)} \in x^{(0)} + \spK^{(k)}_\times(A,r^{(0)})$, the residual $r^{(k)} = r^{(0)} - Ax^{(k)}$ need not be contained in
$\mathcal{K}\ind{\times}{k+1}(A,r^{(0)})$. This fact prevents approaches based on the variational characterization \eqref{full_qgmres:eq} to be realized with cost depending exclusively on $k$ and not on $n$.

As an alternative, we thus suggest an approach similar to truncated GMRES (see \cite{Saad2003}, e.g.).  We project the $n \times d\ind{\times}{k}$ least squares problem \eqref{ls_qgmres:eq} onto a $d\ind{\times}{k+1} \times d\ind{\times}{k}$ least squares problem by minimizing, instead of the whole residual $\|b-Ax^{(k)}\|$, only its orthogonal projection on $\mathcal{K}^{(k+1)}_\times(A,r^{(0)})$.

\begin{definition} The $k$-th {\em quadratic quasi GMRES (``QQGMRES'')} iterate $x^{(k)}_{\qqgmres}$ is the solution of the least squares problem
\begin{equation} \label{ls1:eq}
x^{(k)}_{\qqgmres} = \argmin_{x \in x^{(0)} + \spK\ind{\times}{k}(A,b)} \| (V\ind{\times}{k+1})^*(b - Ax) \|.
\end{equation}
\end{definition}

Computationally, we have that $x^{(k)}_{\qqgmres} = x^{(0)} + V_\times^{(k)} \zeta_k$, where $\zeta_k$ solves the $d\ind{\times}{k+1} \times d\ind{\times}{k}$ least squares problem
\begin{equation} \label{ls2:eq}
\zeta_k = \argmin_{\zeta \in \mathC^{d\ind{\times}{k}}} \| (V^{(k+1)}_\times)^*r^{(0)} - (V^{(k+1)}_\times)^* A V_\times^{(k)} \zeta \|,
\end{equation}
where
\begin{equation} \label{ls_reduced_model:eq}
\underline{H}\ind{\times}{k} = (V\ind{\times}{k+1})^*  A V\ind{\times}{k}  = \begin{bmatrix} (V\ind{1}{k+1})^*A_{11}V\ind{1}{k} & (V\ind{1}{k+1})^*A_{12}V\ind{2}{k} \\ (V\ind{2}{k+1})^*A_{21}V\ind{1}{k} & (V\ind{2}{k+1})^*A_{22}V\ind{2}{k} \end{bmatrix}
\end{equation}
and where the structure of $(V_{k+1}^\times)^*r^{(0)}$ is given in \eqref{e2:eq}. 

\subsection{Analysis of QGMRES and QQGMRES}

As for QFOM, there is no polynomial interpolation property for QGMRES nor for QQGMRES. We can again present only simple first elements of an analysis. 

As solutions to least squares problems, the iterates $x\ind{\qgmres}{k}$ and $x\ind{\qqgmres}{k}$ are always defined. They are uniquely defined in case of QGMRES, since 
$AV\ind{\times}{k}$ has full rank since $V\ind{\times}{k}$ has full rank. For QQMRES we have

\begin{proposition} The matrix $\underline{H}\ind{\times}{k}$ from \eqref{ls_reduced_model:eq} has full rank if $0 \not \in W^2(A)$.
\end{proposition}
\begin{proof}
The matrix $\underline{H}\ind{\times}{k}$ is obtained from $H\ind{\times}{k}$ by complementing it with two rows,  
one after each block, and $H\ind{\times}{k}$ is nonsingular by Theorem~\ref{QFOM_properties:thm}(iii). Thus, $\underline{H}\ind{\times}{k}$ has full rank $d\ind{\times}{k} = d\ind{1}{k} + d\ind{2}{k}$. 
\end{proof}

QGMRES and QQGMRES also both have a finite termination property.
\begin{proposition} Let $k_{\max} \leq g(r\ind{}{0})$ be as in the proof of Theorem~\ref{QFOM_properties:thm}. 
Then $x\ind{\qgmres}{k_{\max}} = A^{-1}b$. Provided $\underline{H}\ind{\times}{k_{\max}}$ has full rank, we also have $x\ind{\qqgmres}{k_{\max}} = A^{-1}b$.
\end{proposition}
\begin{proof} 
As in the proof of Theorem~\ref{QFOM_properties:thm}, we have that
$x^* = A^{-1}b = x^{(0)} + y^*$ with $y^* = A^{-1}r^{(0)}$ being contained in $\spK_\times^{(k_{\max})}(A,r^{(0)})$.
So $x^*$ satisfies the variational characterization~\eqref{full_qgmres:eq} with residual norm 0, and as such it is unique. This implies that $x^*$ is identical to the  QGMRES iterate $x\ind{\qgmres}{k_{\max}}$. For QQGMRES, we write $y^* \in \spK_\times^{(k_{\max})}(A,r^{(0)})$ as $y^* = V\ind{\times}{k}\zeta$. This $\zeta$ is a solution of the least squares problem \eqref{ls2:eq}, yielding the minimal value 0 for the resiudal norm. If $\underline{H}\ind{\times}{k_{\max}}$ has full rank, the solution of the least squares problem \eqref{ls2:eq} is unique. And since the QQGMRES iterate $x\ind{\qqgmres}{(k_{\max}}$ is obtained by solving this least squares problem, it is equal to $x^*$.
\end{proof}

Trivially, QGMRES gets iterates $x\ind{\qgmres}{k}$ whose residuals $r\ind{\qgmres}{k}$ are smaller in norm than 
$r\ind{\gmres}{k}$, i.e.\ the residual of the iterate $x\ind{\gmres}{k}$ of standard GMRES, since QGMRES minimizes the residual norm over a larger subspace. Moreover, since QQGMRES minimizes over the same subspace as QGMRES, but minimizes the norm of the projection of the residual rather than the norm of the residual itself, we also have that $\|r\ind{\qgmres}{k}\| \leq \|r\ind{\qqgmres}{k}\|$. Finally, note that we cannot expect the relation $\|r\ind{\qqgmres}{k}\| \leq \|r\ind{\gmres}{k}\|$ to hold in general.

\section{Algorithmic aspects} \label{algorithm:sec}

An important practical question is how one can compute $V^{(k)}_\times$ and $H^{(k)}_\times$ efficiently and in a stable manner. Interestingly, for the special case where $A_{21} = I$ and $A_{22} = 0$, which arises in the linearization of quadratic eigenvalue problems, this question has been treated in many papers, and recently the {\em two-level orthogonal Arnoldi method} has emerged as a cost-efficient and at the same time stable algorithm; see \cite{KressnerRoman2014, LuSuBai2016, MeerbergenPerez2018}. 
In the following, we describe how the two-level orthogonal Arnoldi method generalizes to general $2\times 2$ block matrices with minor changes. Generalizing the stability analysis is not as straightforward, and a detailed analysis is beyond the scope of this paper. The main idea is that we refrain from directly computing the orthogonal Arnoldi basis $V\ind{}{k}$ from \eqref{Arnoldi_relation:eq}, but rather compute/update the orthonormal bases $V\ind{1}{k},V\ind{2}{k}$ of its block components while at the same time updating $H\ind{\times}{k}$.  

Assume that no breakdown occurs and no deflation is necessary. Then we have (see \eqref{V_relations:eq}) 
\[
V\ind{}{k} =  \begin{bmatrix} V\ind{1}{k}R\ind{1}{k} \\ V\ind{2}{k} R\ind{2}{k} \end{bmatrix},
\]
where the $V\ind{i}{k}$ have $k$ orthonormal columns, and the $R\ind{i}{k} \in \mathC^{k \times k}$ are upper triangular. Since the columns of $V\ind{}{k}$ are orthonormal, too, this implies
\begin{equation} \label{sum_R:eq}
   (R\ind{1}{k})^*R\ind{1}{k} + (R\ind{2}{k})^*R\ind{2}{k} = (V\ind{}{k})^*V\ind{}{k} = I,
\end{equation}
showing that the matrix $\left[ \begin{smallmatrix} R\ind{1}{k} \\ R\ind{2}{k} \end{smallmatrix} \right] \in \mathC^{2k \times k}$ also has orthonormal columns. 
Writing the Arnoldi relation \eqref{Arnoldi_relation:eq} in terms of the block components gives 
\begin{equation} \label{recursion:eq}
    \begin{split}
        A_{11} V\ind{1}{k} R\ind{1}{k} + A_{12} V\ind{2}{k} R\ind{2}{k}
        &= V\ind{1}{k+1} R\ind{1}{k+1} \underline H\ind{}{k} =: V\ind{1}{k+1} \underline{H}\ind{1}{k},  \\
        A_{21} V\ind{1}{k} R\ind{1}{k} + A_{22} V\ind{2}{k} R\ind{2}{k}
        &= V\ind{2}{k+1} R\ind{2}{k+1} \underline H\ind{}{k} =: V\ind{2}{k+1} \underline{H}\ind{2}{k}, 
    \end{split}
\end{equation}
where the matrices
\[
\underline{H}\ind{i}{k} := R\ind{i}{k+1} \underline H\ind{}{k} \in \mathC^{(k+1)\times k},\qquad i=1,2,
\]
are upper Hessenberg.

The relation \eqref{recursion:eq} reveals that $V\ind{i}{k+1}$ can be obtained as an update of $V\ind{i}{k}$ by adding a new last column, and $\underline H\ind{i}{k}$ as an update of $\underline H\ind{i}{k-1}$ by adding a new last column and a new last row. Thus, the new column of $V\ind{i}{k+1}$ arises from the orthonormalization of the last column of $A_{i1} V\ind{1}{k} R\ind{1}{k} + A_{i2} V\ind{2}{k} R\ind{2}{k}$ against all columns of $V\ind{i}{k}$ and it is nonzero. The upper-Hessenberg matrix $\underline H\ind{1}{k}$ is obtained from $\underline H\ind{1}{k-1}$ by first adding a new last row of zeros and then adding a new last column holding the coefficients from the orthonormalization. 
To obtain a viable computational scheme, it remains to show that 
$R\ind{i}{k+1}$ as well as $\underline H\ind{}{k}$ (which we need to get the QFOM or QGMRES iterates) 
can also be obtained from these quantities. We do so by establishing how to get them as updates from  
$H\ind{}{k-1}$ and $R\ind{i}{k}$, noting that in the very first step we have 
\[
R\ind{i}{1} = \|b_i\|,\qquad V\ind{i}{1} = b_i/\|b_i\|,\qquad i=1,2,
\]
unless $b_i = 0$ in which case we let the corresponding $R\ind{i}{1}$ be zero and let $V\ind{i}{1}$ be a random unitary vector.


For $k > 1$ we write
\begin{equation*}
    R\ind{i}{k+1} =
    \begin{bmatrix}
        R\ind{i}{k} & r\ind{i}{k+1} \\
        0 & \rho\ind{i}{k+1}
    \end{bmatrix}
    \quad\text{and}\quad
    \underline H\ind{}{k} =
    \begin{bmatrix}
        \underline H\ind{}{k-1} & h\ind{}{k} \\
        0 & \eta\ind{}{k}
    \end{bmatrix},
\end{equation*}
where $R\ind{i}{k}$ and $\underline H\ind{}{k-1}$ are known, and the remaining quantities are to be determined.
Since $\underline H\ind{i}{k}$ equals
\begin{equation}\label{HK:eq}
    R\ind{i}{k+1} \underline H\ind{}{k}
    = \begin{bmatrix}
        R\ind{i}{k} \underline H\ind{}{k-1} & R\ind{i}{k} h\ind{}{k} + \eta\ind{i}{k} r\ind{i}{k+1} \\
        0 & \eta\ind{}{k} \rho\ind{i}{k+1}
    \end{bmatrix}
    = \begin{bmatrix}
        \underline H\ind{i}{k-1} & h\ind{i}{k} \\
        0 & \eta\ind{i}{k}
    \end{bmatrix},
\end{equation}
it follows, using \eqref{sum_R:eq}, that
\begin{eqnarray*}
    [(R\ind{1}{k})^* \; 0] \underline H \ind{1}{k} + [(R\ind{2}{k})^* \;0] \underline H\ind{2}{k}
    &=& \big( (R\ind{1}{k})^* [R\ind{1}{k}\; r\ind{1}{k+1}] + (R\ind{2}{k})^* [R\ind{2}{k}\; r\ind{2}{k+1}] \big) \underline H\ind{}{k} \\
    &=& [I\;  0] \underline H\ind{}{k}
    \, = \, H\ind{}{k}.
\end{eqnarray*}
Hence, we see that 
\begin{equation} \label{h_update:eq}
h\ind{}{k} = (R\ind{1}{k})^*  h\ind{1}{k} + (R\ind{2}{k})^* h\ind{2}{k}, 
\end{equation}
which allows for the computation of $h\ind{}{k}$ from known quantities.
Once $h\ind{}{k}$ is known, \eqref{HK:eq} can be used to compute
\[
    \widetilde r\ind{i}{k+1} = \eta\ind{}{k} r\ind{i}{k+1} = h\ind{i}{k} - R\ind{i}{k} h\ind{}{k},
\]
at which point $\eta\ind{}{k}$ and the $\rho\ind{i}{k}$ are the only remaining  quantities to be determined. 
Letting $\eta\ind{}{k}$ be real valued (and nonnegative) allows its computation in at least two different ways.
The first is to consider the bottom right entry of \eqref{sum_R:eq} which gives 
\begin{eqnarray*}
    (\eta\ind{}{k})^2
    &=& \|\eta\ind{}{k} r\ind{1}{k+1}\|^2 + |\eta\ind{}{k} \rho\ind{1}{k+1}|^2 + \|\eta\ind{}{k} r\ind{2}{k+1}\|^2 + |\eta\ind{}{k} \rho\ind{2}{k+1}|^2 \\
    &=& \|\widetilde r\ind{1}{k+1}\|^2 + |\eta\ind{1}{k}|^2 + \|\widetilde r\ind{2}{k+1}\|^2 + |\eta\ind{2}{k}|^2.
\end{eqnarray*}
The second possibility is to determine $\eta\ind{}{k}$ from the $(k+1,k+1)$ entry of the equality $(\underline H\ind{}{k})^* \underline H \ind{}{k} = (\underline H\ind{1}{k})^* \underline H\ind{1}{k} + (\underline H\ind{2}{k})^* \underline H\ind{2}{k}$, which results in
\begin{equation*}
    (\eta\ind{}{k})^2 + \| h\ind{}{k}\|^2
    = \|h\ind{1}{k}\|^2 + |\eta\ind{1}{k}|^2 + \| h\ind{2}{k}\|^2 + |\eta\ind{2}{k}|^2,
\end{equation*}
using \eqref{sum_R:eq}. The first method may be preferred, since it guarantees that the computed $(\eta\ind{}{k})$ is nonnegative, even with roundoff errors.
Once $\eta\ind{}{k}$ has been determined, we get $\rho\ind{i}{k}$ as $\rho\ind{i}{k} = \eta\ind{i}{k}/\eta\ind{}{k}$ from \eqref{HK:eq}. Putting everything together yields the following proposition.

\begin{proposition}
In iteration $k$, the quantities $V\ind{i}{k+1}, R\ind{i}{k+1}$ and $\underline H\ind{i}{k}$ as well as $\underline H\ind{}{k}$ can be obtained from those of iteration $k-1$ at cost comparable to one matrix-vector multiplication with $A$, $2k$ vector scalings and additions with vectors of length $n$ and additional $\mathcal O(k^2)$ arithmetic operations.
\end {proposition}
\begin{proof}
Computing the last column of $V\ind{i}{k}R\ind{i}{k}$ costs $k$ vector scalings and additions with vectors of length $n_i$ for $i=1,2$, which is comparable to $k$ scalings and additions with vectors of length $n$. Multiplication of these last columns with the $A_{ij}$ in \eqref{recursion:eq} amounts to one matrix vector multiplication with $A$. Orthogonalizing the two resulting blocks against all columns of $V\ind{k}{i}$ costs again $k$ scalings and additions of vectors of size $n_1$ and $n_2$ which corresponds to additional $k$ such operations on vectors of length $n$. All other necessary updates as described before require $\mathcal{O}(k^2)$ operations. 
\end{proof}


In the standard Arnoldi process, when $\eta\ind{}{k} = 0$, we know that we have reached the maximum size of the Krylov subspace, i.e.\ $k$ is equal to the grade of the initial residual $r\ind{}{0}$, and that $A^{-1}b$ is
contained in $\spK\ind{}{k}(A,r\ind{}{0})$. Since by \eqref{HK:eq} we have $\eta\ind{i}{k} = \rho\ind{i}{k} \eta\ind{}{k}$, $i=1,2$,
we see that the two-level orthogonal Arnoldi method also stops when $\eta\ind{}{k} = 0$. However, the reverse statement need not necessarily be true, i.e.\ we can have $\eta\ind{i}{k} = 0$ for $i=1,2$ without having $\eta\ind{}{k} = 0$. This would represent a serious breakdown of the two-level orthogonal Arnoldi process. Of course, exact zeros rarely appear in a numerical computation, but 
near breakdowns should be dealt with appropriately. In our implementation, we simply chose to replace a block vector corresponding to some $\eta\ind{i}{k} \approx 0$ by a vector with just random entries. This makes the book-keeping much easier, since then $d\ind{i}{k} = k$ for all $k$ and $i=1,2$, while keeping $V\ind{\times}{k}$ as a subspace of our approximation space. 



The full algorithm is summarized in Algorithm~\ref{alg:qkrylov}. We assume no deflation is necessary and no breakdown occurs for simplicity, but we can deal with this in practice in two ways. When $\widetilde v\ind{i}{k+1}$ is (numerically) linear dependent, we can either set $v\ind{i}{k+1}$ to some random vector and set $\eta\ind{i}{k}$ to zero, or we can set $V\ind{i}{k+1} = V\ind{i}{k}$ and $\underline H\ind{i}{k} = [H\ind{i}{k-1}\; h\ind{i}{k}]$. The former approach requires less bookkeeping, but the latter approach can safe space and time. Another simplification compared to a practical implementation is the use of classical Gramm--Schmidt for the orthogonalization, instead of repeated Gram--Schmidt or modified Gram--Schmidt. However, the algorithm does show how to avoid unnecessary recomputation of quantities. In particular, we avoid recomputing matrix-vector products by updating the products $W\ind{ij}{k} = A_{ij} V\ind{j}{k}$, $Z\ind{ij}{k,k} = (V\ind{i}{k})^* A_{ij} V\ind{j}{k}$, and $Z\ind{ij}{k+1,k} = (V\ind{i}{k+1})^* A_{ij} V\ind{j}{k}$.
Since this updating approach requires more memory, it should only be used if that extra memory is available, and if matrix-vector products with $A$ are sufficiently expensive.

\begin{algorithm2e}[!htbp]
    \SetAlgorithmStyle
    \caption{Quadratic Krylov}
    \label{alg:qkrylov}
    \KwIn{$A_{11}$, $A_{12}$, $A_{21}$, $A_{22}$, $b_1$, $b_2$, $k_{\max}$, and $\tau$}
    $\underline H\ind{}{0} = []$ and $\beta = (\|b_1\|^2 + \|b_2\|^2)^{-1/2}$\;
    \For{$i = 1, 2$}{
        $\rho\ind{i}{1} = \|b_i\| / \beta$ and $v\ind{i}{1} = b_i / \rho\ind{i}{1}$\;
        $\underline H\ind{i}{0} = []$, $R\ind{i}{1} = [\rho\ind{i}{1}]$, and $V\ind{i}{1} = [v\ind{i}{1}]$\;
        }
    \For{$k = 1$ to $k_{\max}$}{
        \For(\tcc*[f]{Update matrix products.}){$i = 1, 2$}{
            \For{$j = 1, 2$}{
                $w\ind{ij}{k} = A_{ij} v\ind{j}{k}$\;
                $W\ind{ij}{k} = [W\ind{ij}{k-1}\; w\ind{ij}{k}]$\;
                $Z\ind{ij}{k,k} = [Z\ind{ij}{k,k-1}\; (V\ind{i}{k})^* w\ind{ij}{k}]$\;
            }
        }
        \For(\tcc*[f]{Update $V\ind{i}{k+1}$ and $\underline H\ind{i}{k}$.}){$i = 1, 2$}{
            $\widetilde v\ind{i}{k+1} = W\ind{i1}{k} (R\ind{1}{k} e\ind{}{k}) + W\ind{i2}{k} (R\ind{2}{k} e\ind{}{k})$\;
            $h\ind{i}{k} = (V\ind{i}{k})^* \widetilde v\ind{i}{k+1}$\;
            $\eta\ind{i}{k} = \|\widetilde v\ind{i}{k+1} - V\ind{i}{k}h\ind{h}{k}\|$\;
            $v\ind{i}{k+1} = (\widetilde v\ind{i}{k+1} - V\ind{i}{k}h\ind{h}{k}) / \eta\ind{i}{k}$\;
            $V\ind{i}{k+1} = [V\ind{i}{k}\; v\ind{i}{k+1}]$ and $\underline H\ind{i}{k} = \Big[ \begin{smallmatrix} H\ind{i}{k-1} & h\ind{i}{k} \\ 0^T & \eta\ind{i}{k} \end{smallmatrix} \Big]$\;
            \For{$j = 1, 2$}{
                $Z\ind{ij}{k+1,k} = [Z\ind{ij}{k,k};\; (v\ind{i}{k+1})^* W\ind{ij}{k}]$\;
            }
        }
        \tcc{Update $\underline H\ind{}{k}$ and $R\ind{i}{k+1}$.}
        $h\ind{}{k} = (R\ind{1}{k})^* h\ind{1}{k} + (R\ind{2}{k})^* h\ind{2}{k}$\;
        \For{$i = 1, 2$}{
            $\widetilde r\ind{i}{k+1} = h\ind{i}{k} - R\ind{i}{k} h\ind{}{k}$\;
        }
        $\eta\ind{}{k} = (\|\widetilde r\ind{i}{k+1}\|^2 + |\eta\ind{1}{k}|^2 + \|\widetilde r\ind{i}{k+1}\|^2 + |\eta\ind{2}{k}|^2)^{1/2}$\;
        \For{$i = 1, 2$}{
            $r\ind{i}{k+1} = \widetilde r\ind{i}{k+1} / \eta\ind{}{k}$, $\rho\ind{i}{k+1} = \eta\ind{i}{k} / \eta\ind{}{k}$, and $R\ind{i}{k+1} = \Big[ \begin{smallmatrix} R\ind{i}{k} & r\ind{i}{k+1} \\ 0 & \rho\ind{i}{k+1} \end{smallmatrix} \Big]$\;
        }
        \tcc{Compute the approximation $x\ind{\texttt{qfom}}{k}$ and the residual $r\ind{\texttt{qfom}}{k}$.}
        $H\ind{\times}{k} = \Big[ \begin{smallmatrix} Z\ind{11}{k,k} & Z\ind{12}{k,k} \\ Z\ind{21}{k,k} & Z\ind{22}{k,k} \end{smallmatrix} \Big]$ and $b\ind{\times}{k} = \beta \Big[ \begin{smallmatrix} R\ind{1}{k} e\ind{}{k} \\ R\ind{2}{k} e\ind{}{k} \end{smallmatrix} \Big]$\;
        $\Big[ \begin{smallmatrix} c\ind{\texttt{qfom}}{k} \\ d\ind{\texttt{qfom}}{k} \end{smallmatrix} \Big] = (H\ind{\times}{k})^{-1} b\ind{\times}{k}$ and $x\ind{\texttt{qfom}}{k} = \Big[ \begin{smallmatrix} V\ind{1}{k} c\ind{\texttt{qfom}}{k} \\ V\ind{2}{k} d\ind{\texttt{qfom}}{k} \end{smallmatrix} \Big]$\;
        $r\ind{\texttt{qfom}}{k} = \Big[ \begin{smallmatrix} b_1 - W\ind{11}{k} c\ind{\texttt{qfom}}{k} - W\ind{12}{k} d\ind{\texttt{qfom}}{k} \\ b_2 - W\ind{21}{k} c\ind{\texttt{qfom}}{k} - W\ind{22}{k} d\ind{\texttt{qfom}}{k} \end{smallmatrix} \Big] $\;
        \If{$\|r\ind{\texttt{qfom}}{k}\| \le \tau \beta$}{
            \KwRet $x\ind{\texttt{qfom}}{k}$\;
        }
    }
    \KwRet $x\ind{\texttt{qfom}}{k_{\max}}$
\end{algorithm2e}

From the pseudocode of the algorithm we can determine the computational cost per iteration as follows.
We count one matrix-vector multiplication with each of the blocks $A_{11}$, $A_{12}$, $A_{21}$, and $A_{22}$, which equals one matrix-vector multiplication with $A$.
Then we have an orthogonalization cost of $\mathcal O((n_1 + n_2)k) = \mathcal O(nk)$, which equals the orthogonalization cost in the standard Arnoldi process.
Updating the $Z_{ij}$ costs $\mathcal O(nk)$ floating-point operations per iteration, but does not have an equivalent cost in Arnoldi.
The same is true for updating the matrices $\underline H\ind{}{k}$ and $R\ind{i}{k+1}$ for $i = 1,2$, although the cost is limited to $\mathcal O(k)$ flops in this case.
Computing $c\ind{\texttt{qfom}}{k}$ and $d\ind{\texttt{qfom}}{k}$ takes $\mathcal O(k^3)$ floating-point operations, while computing the approximation $x\ind{\texttt{qfom}}{k}$ and its residual $r\ind{\texttt{qfom}}{k}$ require $\mathcal O(nk)$.
Clearly, computing the approximation and its residual is expensive, but there is no need to do it in every iteration.
For example, in a restarted version of the QFOM algorithm, we may decide to compute them only once per restart, after the inner loop reaches $k_{\max}$.
When we add everything together, we see that QFOM has the same asymptotic cost as FOM, although QFOM does require more memory.

With minor changes, we can change the code of Algorithm~\ref{alg:qkrylov} to compute the QQGMRES approximation instead of the QFOM approximation.
One downside of QQGMRES is that we cannot guarantee that its approximation, or even the residual norm of its approximation, is better than that of GMRES.
We can remedy this problem by interpolating between the GMRES and the QQGMRES solution.
Let $r\ind{\texttt{gmres}}{k} = b - A x\ind{\texttt{gmres}}{k}$ and $r\ind{\texttt{qqgmr}}{k} = b - A x\ind{\texttt{qqgmr}}{k}$, then
\begin{multline*}
    \|b - A (\alpha x\ind{\texttt{gmres}}{k} + (1-\alpha) x\ind{\texttt{qqgmr}}{k})\|^2
    = \|\alpha r\ind{\texttt{gmres}}{k} + (1-\alpha) r\ind{\texttt{qqgmr}}{k}\|^2 \\
    = \alpha^2 \|r\ind{\texttt{gmres}}{k} - r\ind{\texttt{qqgmr}}{k}\|^2 + 2\alpha (\Re\{(r\ind{\texttt{gmres}}{k})^* r\ind{\texttt{qqgmr}}{k}\} - \|r\ind{\texttt{qqgmr}}{k}\|^2) + \|r\ind{\texttt{qqgmr}}{k}\|^2.
\end{multline*}
Hence, the residual norm of the interpolated approximation is minimized for
\begin{equation*}
  \alpha_{\texttt{opt}} = \frac{\|r\ind{\texttt{qqgmr}}{k}\|^2 - \Re\{(r\ind{\texttt{gmres}}{k})^* r\ind{\texttt{qqgmr}}{k}\}}{\|r\ind{\texttt{gmres}}{k} - r\ind{\texttt{qqgmr}}{k}\|^2}
\end{equation*}
if $r\ind{\texttt{gmres}}{k} \neq r\ind{\texttt{qqgmr}}{k}$.
The residual norm of the approximation $x\ind{\texttt{opt}}{k}$ corresponding $\alpha_\texttt{opt}$ is
\begin{equation*}
  \|r_{\texttt{opt}}\|^2
  = \frac{\|r\ind{\texttt{gmres}}{k}\|^2 \|r\ind{\texttt{qqgmr}}{k}\|^2 - \Re\{(r\ind{\texttt{gmres}}{k})^* r\ind{\texttt{qqgmr}}{k}\}^2}{\|r\ind{\texttt{gmres}}{k} - r\ind{\texttt{qqgmr}}{k}\|^2},
\end{equation*}
and satisfies $\|r_{\texttt{opt}}\| \le \min \{ \|r_{\texttt{gmres}}\|, \|r_{\texttt{qqgmr}}\| \}$.

\section{Numerical experiments} \label{numerics:sec}


\subsection{The Hain-L\"ust operator}
Hain-L\"ust operators appear in magnetohydrodynamics \cite{HainLuest58},
and their spectral properties, in particular their quadratic numerical range, were investigated in a series of papers, e.g.,
in \cite{LangerTretter1998,MuhammadMarletta12,MuhammadMarletta13}.
We consider the Hain-L\"ust operator 
\[
\spA = \begin{bmatrix} -\spL & I \\ I & q \end{bmatrix}
\]
acting on $L^2([0,1])\times L^2([0,1])$
where $\spL=d^2/dx^2$ is the Laplace operator on $[0,1]$ with Dirichlet boundary conditions, $I$ is the identity operator,
and $q$ denotes multiplication by the function $q(x)=-3+2e^{2\pi ix}$.
The domain of $\spA$ is
$D(\spA)=(H^2([0,1])\cap H^1_0([0,1])) \times L^2([0,1])$.

We consider a discretization of $\spA$, approximating function values at an equispaced grid for both blocks, i.e.\ we take 
$x_j=jh$, $j=0,\dots,N+1$, $h=1/(N+1)$ and obtain, using finite differences, the discretized Hain-L\"ust operator 
\[
A = \begin{bmatrix} \frac{1}{h^2}L & I \\ I & Q \end{bmatrix} \in \mathC^{2N\times 2N},
\]
with $L = \mbox{tridiag}(-1,2,-1) \in \mathC^{N \times N}$ and 
$ Q = -3I + 2\diag(e^{2h\pi i}, \dots, e^{2hN\pi i}) \in \mathC^{N \times N}$,
see \cite{MuhammadMarletta13} for more details.

Note that $\tfrac{1}{h^2}L$ is Hermitian and that $Q$ is normal, so the numerical ranges of these diagonal blocks of $A$ satisfy
\begin{eqnarray*}
W_1 &:=&  W(\tfrac{1}{h^2} L ) \, =\,  \tfrac{1}{h^2}[2-2\cos(\pi h), 2 + 2\cos(\pi h)] \, =: \, [\alpha_{\min}(h), \alpha_{\max}(h)], \\
W_2 &=& W(Q) \, =\, \conv \{-3+2e^{2\pi h j}, j=1,\ldots,N\} \subseteq C(-3,2),
\end{eqnarray*}
where $C(-3,2)$ is the circle with center $-3$ and radius $2$. Since both numerical ranges $W_1$ and $W_2$ are contained in the convex set $W(A)$ we see that $0 \in W(A)$. The following argumentation shows that, with the possible exception of very large values for $h$, we have $0 \not \in W^2(A)$: Any $\lambda \in W^2(A)$ 
satisfies
\begin{equation} \label{2x2:eq}
(\lambda-x_1^*\tfrac{1}{h^2}Lx_1)(\lambda-x_2^*Qx_2) = (x_1^*x_2)(x_2^*x_1),
\end{equation}
for some $x_1,x_2$ with $\|x_1\| = \|x_2\| = 1$.
Assume that $\lambda$ lies within the strip $a < \Re(\lambda) < b$ with $-1 < a <0$ and $0<b< \alpha_{\min}(h)$. Then we have $\cw{d(\lambda,W_1) > \alpha_{\min}(h)-b}$ as well as  $\cw{d(\lambda,W_2) > a+1}$ for the distances of $\lambda$ to the sets $W_1,W_2$. Taking absolute values in \eqref{2x2:eq} and using the bound $|x^*_1x_2| \leq 1$ we thus see that $\lambda$ from this strip cannot be in $W^2(A)$ if $(a+1)(\alpha_{\min}(h)-b) > 1$. This is the case, for example, if 
$b < \alpha_{\min}(h) -2$ and $a > -\tfrac{1}{2}$. Note that $\lim_{h \to 0}\alpha_{\min}(h) = \pi^2$.

\begin{figure}
  \centering
  \begin{tikzpicture}
    \begin{semilogyaxis}[
          height=0.475\textwidth
        , width=0.475\textwidth
        , xlabel={Restart}
        , ylabel={Relative residual norm}
        , every axis legend/.append style = {
              anchor = north east
            , at = {(0.98, 0.98)}
          }
        , enlarge x limits = false
        ]

        \addlegendentry{FOM}
        \addplot+ table [x=RST, y=FOM] {\hltable};
        
        \addlegendentry{GMRES}
        \addplot+ table [x=RST, y=GMRES] {\hltable};
        
        \addlegendentry{QFOM}
        \addplot+ table [x=RST, y=QFOM] {\hltable};

        \addlegendentry{QQGMRES}
        \addplot+ table [x=RST, y=QQGMRES] {\hltable};
        
        \addlegendentry{Interp}
        \addplot+ table [x=RST, y=InterpQQGMRES] {\hltable};
    \end{semilogyaxis}
  \end{tikzpicture}
  \begin{tikzpicture}
    \begin{semilogyaxis}[
          height=0.475\textwidth
        , width=0.475\textwidth
        , xlabel={Restart}
        , ylabel={Relative residual norm}
        , every axis legend/.append style = {
              anchor = north east
            , at = {(0.98, 0.98)}
          }
        , enlarge x limits = false
        , yticklabel pos = right
        ]

        \addlegendentry{FOM}
        \addplot+ table [x=RST, y=FOM] {\hltablelarge};
        
        \addlegendentry{GMRES}
        \addplot+ table [x=RST, y=GMRES] {\hltablelarge};
        
        \addlegendentry{QFOM}
        \addplot+ table [x=RST, y=QFOM] {\hltablelarge};

        \addlegendentry{QQGMRES}
        \addplot+ table [x=RST, y=QQGMRES] {\hltablelarge};
        
        \addlegendentry{Interp}
        \addplot+ table [x=RST, y=InterpQQGMRES] {\hltablelarge};
    \end{semilogyaxis}
  \end{tikzpicture}
  \caption{Convergence plots for the discretized Hain-L\"ust operator: $N=1\,023$ (left) and  $N = 16\,383$ (right), \label{fig:hl1}}
\end{figure}
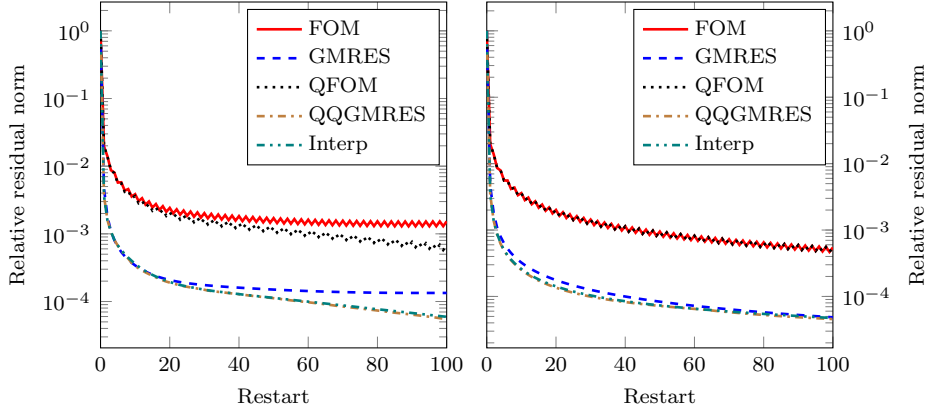

In all our examples we chose the right hand side $b$ as $b = Ae$ where $e$ is 
the vector of all ones, and our initial guess is always $x^0 = 0$. 
Figure~\ref{fig:hl1} shows convergence plots for FOM, GMRES, QFOM, QQGMRES and
the interpolated QQGMRES method as described at the end of Section~\ref{algorithm:sec}.  
The figure displays the relative norm of the residual 
as a function of the invested matrix-vector multiplications. In the left part,
we took $N=1\,023$, the right part is for $N = 16\,383$.  We restarted every method after $m=50$ iterations to avoid that the arithmetic work and the storage 
related with the (two-level) Arnoldi process becomes too expensive.  Note that the figure displays the residual norms at the end of each cycle only, which makes the convergence of some of the methods, in particular FOM, to appear smoother than it actually is. Two major 
observations can be made: On the one side, the FOM type methods yield significantly larger residals than the GMRES type methods. For $N=1\,023$, the ``quadratic methods'' still make progress in the later cycles while their ``non-quadratic'' counter parts then basically stagnate. There is no such difference visible for dimension $N=16\,383$; convergence for all methods is very slow. 

In a second numerical experiment we therefore report results of a geometric
multigrid method as an attempt to cope with large condition
numbers. For a given discretization with step size $h=1/(N+1)$ with $N+1 =
2^k$ we construct the system at the next coarser level to be the discretizaton
with $h_c = 2h = 1/(N_c+1)$ with $N_c+1 = 2^{k-1}$. We stop descending the
grid hierarchy when we reach $N=7$, where we solve the corresponding $14 \times 14$ system by 
explicit inversion of $A$. Interpolation between two levels of the grid
hierarchy is done using standard linear interpolation from the neighboring 
grid points; restriction is the standard adjoint of interpolation. For the
smoothing iteration we test one or five steps of standard GMRES versus one or
two steps of QFOM. We always performed V-cycles with pre-smoothing. The left
part of Figure~\ref{hl2:fig} gives the resulting convergence plots for the
multigrid methods for $N=1\,023$, the right part for $N=16\,383$.

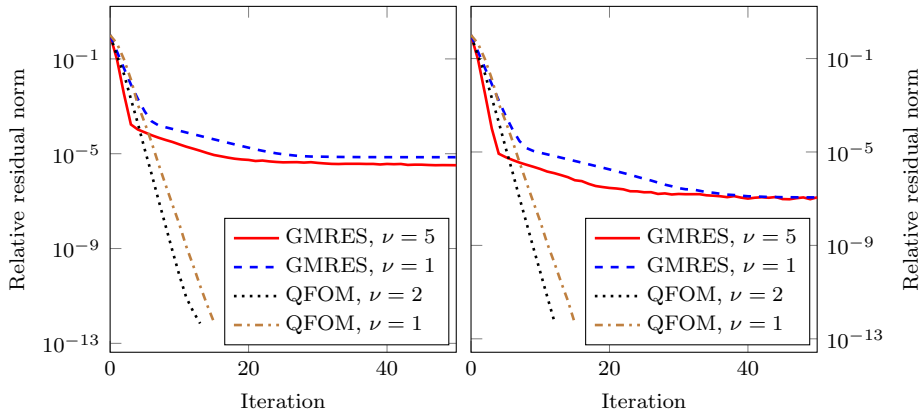
\begin{figure}[htbp!]
  \centering
  \begin{tikzpicture}
    \begin{semilogyaxis}[
          height=0.475\textwidth
        , width=0.475\textwidth
        , xlabel={Iteration}
        , ylabel={Relative residual norm}
        , every axis legend/.append style = {
              anchor = south east
            , at = {(0.98, 0.02)}
          }
        , enlarge x limits = false
        ]

        \addlegendentry{GMRES, $\nu=5$}
        \addplot+ table [x=it, y=gmresfive] {\hlmgsmall};
        
        \addlegendentry{GMRES, $\nu=1$}
        \addplot+ table [x=it, y=gmresone] {\hlmgsmall};
        
        \addlegendentry{QFOM, $\nu = 2$}
        \addplot+ table [x=it, y=qfomtwo] {\hlmgsmall};

        \addlegendentry{QFOM, $\nu = 1$}
        \addplot+ table [x=it, y=qfomone] {\hlmgsmall};
    \end{semilogyaxis}
  \end{tikzpicture}
  \begin{tikzpicture}
    \begin{semilogyaxis}[
          height=0.475\textwidth
        , width=0.475\textwidth
        , xlabel={Iteration}
        , ylabel={Relative residual norm}
        , every axis legend/.append style = {
              anchor = south east
            , at = {(0.98, 0.02)}
          }
        , enlarge x limits = false
        , yticklabel pos = right
        ]

        \addlegendentry{GMRES, $\nu = 5$}
        \addplot+ table [x=it, y=gmresfive] {\hlmglarge};
        
        \addlegendentry{GMRES, $\nu = 1$}
        \addplot+ table [x=it, y=gmresone] {\hlmglarge};
        
        \addlegendentry{QFOM, $\nu = 2$}
        \addplot+ table [x=it, y=qfomtwo] {\hlmglarge};

        \addlegendentry{QFOM, $\nu = 1$}
        \addplot+ table [x=it, y=qfomone] {\hlmglarge};
        
    \end{semilogyaxis}
  \end{tikzpicture}
  \caption{Convergence plots for geometric multigrid for the Hain-L\"ust operator for QFOM and GMRES smoothing and different numbers of smoothing steps $\nu$; $N = 1\,023$ (left), $N = 16\,383$ (right).}
    \label{hl2:fig}
\end{figure}

From these 
plots it is apparent that QFOM is a well-working smoothing iteration for the
multigrid method, whereas GMRES is not, even not for larger numbers of smoothing
steps per iteration. As a complement to these results, Figure~\ref{hl3:fig} illustrates the mesh size independence of the convergence behavior of the multigrid method with QFOM smoothing. It shows that the number of iterations required to reduce the initial residual by a factor of $10^{-12}$ is basically independent of $h$. 
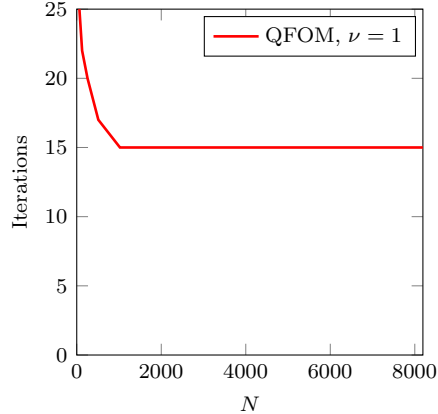
\begin{figure}
    \centering
    \begin{tikzpicture}
        \begin{axis}[
              height=0.475\textwidth
            , width=0.475\textwidth
            , xlabel={$N$}
            , ylabel={Iterations}
            , every axis legend/.append style = {
              anchor = north east
            , at = {(0.98, 0.98)}
          }
        , enlarge x limits = false
         , xmin = 0
         , xmax = 8192
         , ymax = 25
         , ymin = 0
         , ytick = {0,5,10,15,20,25}
        ]

        \addlegendentry{QFOM, $\nu = 1$}
        \addplot+ table [x=N, y=it] {\hlmgscaling};
        
    \end{axis}
  \end{tikzpicture}
    \caption{Number of multigrid iterations needed to reduce the initial residual by a factor of $10^{-12}$ as a function of $N$}
    \label{hl3:fig}
\end{figure}

\subsection{The Schwinger model}
Our second example is the Schwinger model in two-dimensions that arises in computations of quantum electrodynamics (QED). QED models the interactions of electrons and photons and is oftentimes used as a simpler model problem for the $4$-dimensional problems of quantum chromodynamics (QCD).
It is a quantum field theory, meaning that physical quantities arise as expected values of solutions of partial differential equations whose coefficients are coming from the quantum background field, i.e., they are stochastic quantities obeying a given distribution. The Schwinger model is a discretization of the Dirac equation 
\[
\mathcal{D}\psi = \left(\sigma_1 \otimes \left(\partial_x + A_x\right) + \sigma_2 \otimes \left(\partial_y + A_y \right)\right)\psi = \varphi,
\] on a regular, 2-dimensional $N \times N$ cartesian lattice, where the spin structure\footnote{The $\sigma$-matrices are generators of a Clifford algebra and arise in the derivation of the Dirac equation from the Klein-Gordon equation. They give rise to the internal spin (i.e., angular momentum) degrees of freedom of the fields $\psi$ \cite{GattLang2010}. Note that although our discussion is limited to this particular choice of generators, all the results that follow extend to any other of the admissible choices of the $\sigma$-matrices.} is encoded by the Pauli matrices
\[
\sigma_1 = \begin{pmatrix} & 1 \\ 1 & \end{pmatrix},\ \ \sigma_2 = \begin{pmatrix}  & i \\ -i & \end{pmatrix} \text{\ \ and\ \ } \sigma_3 = \begin{pmatrix} 1 & \\ & -1 \end{pmatrix}
\] and $A_{\mu}$ encodes the background gauge field. In the Schwinger model we have $A_{\mu} \in \mathbb{R}$. Using a central covariant finite difference discretization for the first order derivatives, and introducing a scaled second-order stabilization term one writes the action of the discretized operator $D \in \mathbb{C}^{2N^2 \times 2N^2}$ of the Schwinger model at any lattice site $x$ on a spinor $\psi(x) \in \mathbb{C}^{2}$ as
\begin{equation} \label{schwinger1:eq}
\left. 
\begin{array}{rcl} \displaystyle
    \left(D\psi\right)(x) & = & \left(m_0 + 2\right)\psi(x) \\ 
    \displaystyle &  & \displaystyle \mbox{} + \frac{1}{2}\sum_{\mu \in \{x,y\}} \left(\left(I-\sigma_{\mu}\right)\otimes U_{\mu}(x)\right)\psi(x+e_{\mu}) \\ 
   \displaystyle  &  & \displaystyle\mbox{} + \frac{1}{2}\sum_{\mu \in \{x,y\}} \left(\left(I+\sigma_{\mu}\right)\otimes \overline{U_{\mu}(x-e_{\mu})}\right)\psi(x-e_{\mu}).
\end{array}
\right\}
\end{equation}
In here {$U_\mu$ correspond to a discrete version of the stochastically varying gauge field with $ U_\mu(x) \in \mathbb{C}, |U_\mu(x)| = 1$ for all $x$}, and $m_0$ sets the mass of the simulated theory. The naming convention of this formula is depicted in~\cref{fig:namingconvention}, and {we refer to the textbook \cite{GattLang2010}, e.g., for further details.}
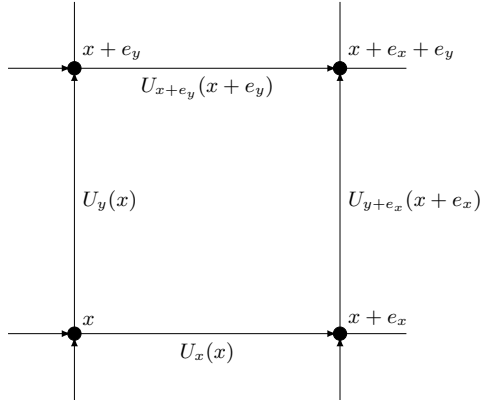
\begin{figure}
    \centering
    \resizebox{.5\linewidth}{!}{
    \begin{tikzpicture}
        \draw[fill=black] (0,0) circle (.1cm) node[above right] (P00) {\small$x$};
        \draw[fill=black] (4,0) circle (.1cm) node[above right] (P10) {\small$x+e_x$};
        \draw[fill=black] (0,4) circle (.1cm) node[above right] (P01) {\small$x+e_y$};
        \draw[fill=black] (4,4) circle (.1cm) node[above right] (P11) {\small$x+e_x+e_y$};
        \draw[-latex] (.05,0) -- node[below] {\small $U_x(x)$} (3.95,0);
        \draw[-latex] (.05,4) -- node[below] {\small $U_{x+e_y}(x+e_y)$} (3.95,4);
        \draw[-latex] (0,.05) -- node[right] {\small $U_y(x)$} (0,3.95);
        \draw[-latex] (4,.05) -- node[right] {\small $U_{y+e_x}(x+e_x)$} (4,3.95);
        \draw[-latex] (-1,0) -- (-.05,0);
        \draw[-latex] (-1,4) -- (-.05,4);
        \draw[-] (4.05,0) -- (5,0);
        \draw[-] (4.05,4) -- (5,4);
        \draw[-latex] (0,-1) -- (0,-.05);
        \draw[-latex] (4,-1) -- (4,-.05);
        \draw[-] (0,4.05) -- (0,5);
        \draw[-] (4,4.05) -- (4,5);
    \end{tikzpicture}
    }
    \caption{Naming conventions in the Schwinger model.}
    \label{fig:namingconvention}
\end{figure}
{The canonical $2\times 2$ block structure of the Schwinger model matrix arises from the spin structure: We reorder the unknowns in $\psi$ according to spin, i.e., we take  
\[
\psi = \begin{pmatrix} \psi_1 \\ \psi_2 \end{pmatrix},
\] 
where $\psi_1 \in \mathbb{C}^{N^2}$ collects all the spin 1 components $\psi_1(x)$ of $\psi(x) = \left[ \begin{smallmatrix}  \psi_1(x) \\ \psi_2(x) \end{smallmatrix} \right] \in \mathbb{C}^2$ at all lattice sites, and similarly for $\psi_2$. Then the reordered discretized Schwinger model matrix, acting on the reordered vector $\left( \begin{smallmatrix}  \psi_1(x) \\ \psi_2(x) \end{smallmatrix} \right) $, is given as 
\[
 D = \begin{pmatrix} A & B \\ -B^{*} & A \end{pmatrix}.
\] 
}

Here, the diagonal blocks $A$ correspond to the discretized second order stabilization term and are thus called gauge Laplace operators, while the off-diagonal blocks $B$ correspond to the central finite covariant difference discretization of the Dirac equation. Using \eqref{schwinger1:eq} we see that the action of {the blocks $A$ and $B$ on a vector $\psi_1, \psi_2$ is given as}
\begin{align*}
  (A \psi_{1})(x) &= (m_0 + 2) \psi_{1}(x)
  \begin{aligned}[t]
    &- \frac{1}{2} \sum_{\mu \in \{x,y\} } U_{\mu}(x) \psi_1(x+e_{\mu}) \\
    &- \frac{1}{2} \sum_{\mu \in \{x,y\} } \overline{U_{\mu}(x-e_{\mu})}\psi_1(x-e_{\mu}),
  \end{aligned}
  \\
  (B \psi_{2})(x) &=
  \begin{aligned}[t]
    &- \frac{1}{2} \left( U_{x}(x) \psi_1(x+e_{x}) + i \cdot U_{y}(x) \psi_1(x+e_{y})\right)\\
    &+ \frac{1}{2} \left(\overline{U_{x}(x-e_{x})}\psi_1(x-e_{x}) - i \cdot \overline{U_{y}(x-e_{y})} \psi_1(x-e_{y})\right).
  \end{aligned}
\end{align*} 
From this we see that the mass parameter $m_0$ induces a shift by a multiple of the identity in $A$, which we make explicit in writing $A = A_0 + m_0I$.

In our tests we consider the ``symmetrized'' operator  $Q := \Sigma_3 D$ with $\Sigma_3 = \sigma_3 \otimes I_{N \cdot N}$. Due to $A^{*} = A, B^{*} = -B$ this operator 
\[
Q = \begin{pmatrix} A & B \\ B^* & -A \end{pmatrix} = \begin{pmatrix} A_0 + m_0I & B \\ B^* & -A_0-m_0I \end{pmatrix} 
\]
is hermitian, but indefinite.

The quadratic range $W_2(Q)$ has two connected components to the left and right of $0$ on the real axis, provided $m_0 > -\alpha_{\min}$, the smallest eigenvalue of $A_0$. This can be seen as follows: Let $x_1, x_2 \in \mathC^{N \times N }$ be two normalized vectors and let
\[
\begin{pmatrix} x_1^*Ax_1 & x_1^*Bx_2 \\ x_2^*B^*x_1 & -x_2^*Ax_2 \end{pmatrix} =: \begin{pmatrix} \alpha_1 & \beta \\ \overline{\beta} & -\alpha_2 \end{pmatrix}.
\]
Then any eigenvalue $\lambda$ of this matrix satisfies 
\begin{eqnarray*}
&& (\lambda-\alpha_1)(\lambda+\alpha_2) \, = \,  |\beta|^2 \\ 
&\Longrightarrow & (\Re(\lambda)-\alpha_1)(\Re(\lambda)+\alpha_2) = |\beta|^2 + \Im(\lambda)^2.
\end{eqnarray*}
The last equality cannot be satisfied if $-\alpha_2 < \Re(\lambda) < \alpha_1$. In particular, if $m_0 > -\alpha_{\min}$, the equality cannot be satisfied if $|\Re(\lambda)| < m_0 +\alpha_{\min}$, since $\alpha_1, \alpha_2 \geq  m_0 +\alpha_{\min}$.

For our tests we use a gauge configuration obtained by a heatbath algorithm excluding the fermionic action, which results in the smallest eigenvalue $\alpha_{\min}$ of $A_0$ being approximately $0.11$.  \cref{fig:schwinger}  reports results for two different choices of $m_0$. As in the first example we perform a restart after every $50$ iterations. The first choice for $m_0$ is $m_0 = -0.1 > -\alpha_{\min}$, so that the quadratic range indeed has two connected components with a gap around $0$. The second is $m_0 = -0.22 < -\alpha_{\min}$, so that $W^2(Q)$ consists of only one component containing $0$.    
The figure shows that  a marked improvement can be observed for the ``quadratic'' methods if the quadratic range consists indeed of two different connected components (left plot), whereas this advantage is lost to a large extent for the second choice for $m_0$,  where $W^2(Q)$ does not indicate a spectral gap (right plot). In this case, the system is also severely ill-conditioned, so that the convergence of all methods considered is much slower.  We also note that for this example and for both choices for $m_0$, interpolated QQGMRES does not differ substantially from standard GMRES. Without showing the corresponding convergence plots, let us at least mention that when decreasing $m_0$ from $-0.1$ to $-0.22$ we observe for a long time a convergence behavior very similar to that for the largest value $-0.1$, even when $m_0$ is already smaller than $-\alpha_{\min}$.

\begin{figure}[htbp!]
  \centering
  \begin{tikzpicture}
    \begin{semilogyaxis}[
          height=0.475\textwidth
        , width=0.475\textwidth
        , xlabel={Restart}
        , ylabel={Relative residual norm}
        , every axis legend/.append style = {
              anchor = north east
            , at = {(0.98, 0.98)}
          }
        , enlarge x limits = false
        ]

        \addlegendentry{FOM}
        \addplot+ table [x=RST, y=FOM] {\schwingertable};
        
        \addlegendentry{GMRES}
        \addplot+ table [x=RST, y=GMRES] {\schwingertable};
        
        \addlegendentry{QFOM}
        \addplot+ table [x=RST, y=QFOM] {\schwingertable};

        \addlegendentry{QQGMRES}
        \addplot+ table [x=RST, y=QQGMRES] {\schwingertable};
        
        \addlegendentry{Interp}
        \addplot+ table [x=RST, y=InterpQQGMRES] {\schwingertable};
    \end{semilogyaxis}
  \end{tikzpicture}
  \begin{tikzpicture}
    \begin{semilogyaxis}[
          height=0.475\textwidth
        , width=0.475\textwidth
        , xlabel={Restart}
        , ylabel={Relative residual norm}
        , every axis legend/.append style = {
              anchor = north east
            , at = {(0.98, 0.98)}
          }
        , enlarge x limits = false
        , yticklabel pos = right
        ]

        \addlegendentry{FOM}
        \addplot+ table [x=RST, y=FOM] {\schwingertablecond};
        
        \addlegendentry{GMRES}
        \addplot+ table [x=RST, y=GMRES] {\schwingertablecond};
        
        \addlegendentry{QFOM}
        \addplot+ table [x=RST, y=QFOM] {\schwingertablecond};

        \addlegendentry{QQGMRES}
        \addplot+ table [x=RST, y=QQGMRES] {\schwingertablecond};
        
        \addlegendentry{Interp}
        \addplot+ table [x=RST, y=InterpQQGMRES] {\schwingertablecond};
    \end{semilogyaxis}
  \end{tikzpicture}
  \caption{Convergence plots for the Schwinger model, $\alpha_{\min} \approx 0.11$, $N = 128^2$. Left: $m_0 = -0.1 > -\alpha_{\min}$, right: $m_0 = -0.22 < -\alpha_{\min}$.}\label{fig:schwinger}
\end{figure}
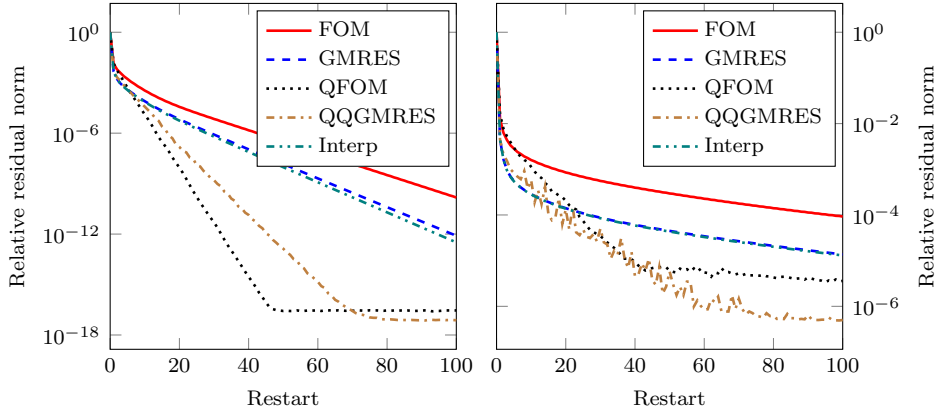

\bibliographystyle{siam}
\bibliography{qfom}

\end{document}